\documentclass[reqno, 10pt]{amsart}
\usepackage[utf8]{inputenc}
\usepackage[english]{babel}
\usepackage[T1]{fontenc}
\usepackage{comment}
\usepackage{amssymb,amsmath,amsfonts,amsthm,mathtools} 
\usepackage{enumitem}
\usepackage{bbm}
\usepackage{graphicx}
\usepackage{latexsym}
\usepackage{float}
\usepackage{pstricks,pst-node,pst-blur,pstricks-add}
\usepackage{pst-all}
\usepackage{pst-plot}
\usepackage{pst-ode}
\usepackage{pst-3dplot}
\usepackage{pstricks-add}
\usepackage{pst-node}
\usepackage{array}
\usepackage{color}
\usepackage{caption,subcaption}	

\usepackage[a4paper,top=3cm,bottom=2cm,left=3cm,right=3cm,marginparwidth=1.75cm]{geometry}

\usepackage[colorinlistoftodos]{todonotes}
\usepackage[colorlinks=true, allcolors=blue]{hyperref}

\newtheorem{theorem}{Theorem}[]
\newtheorem{thmx}{Theorem}

\newtheorem{corollary}[]{Corollary}
\newtheorem{lemma}[]{Lemma}
\newtheorem{proposition}[]{Proposition}
\newtheorem{definition}[]{Definition}
\newtheorem{remark}[]{Remark}

\def\N{\mathbb{N}}
\def\Z{\mathbb{Z}}
\def\R{\mathbb{R}}
\def\T{\mathbb{T}}
\def\B{\mathcal{B}}

\def\F{\mathcal{F}}

\def\M{\mathcal{M}}
\def\P{\mathcal{P}}
\def\Q{\mathcal{Q}}

\def\tH{\tilde{H}}

\def\tx{\tilde{x}}
\def\tf{\tilde{f}}

\def\diam{\operatorname{diam}}
\def\id{\operatorname{id}}

\def\supp{\operatorname{supp}}
\def\quand{\quad\text{and}\quad}

\begin{document}


\title{EQUILIBRIUM STATES FOR MAPS ISOTOPIC TO ANOSOV}
\author{Carlos F. Álvarez}
\address{Instituto de Matemática, Estatística e Computação Científica, IMECC-UNICAMP, Campinas-SP, Brazil}
\email{fabmath92@gmail.com}
\date{2020}

\author{Adriana Sánchez}
\address{Centro de investigación de Matemática Pura y Aplicada, Universidad de Costa Rica. San José, Costa Rica.}
\email{adriana.sanchez\_c@ucr.ac.cr}

\author{Régis Varão}
\address{Instituto de Matemática, Estatística e Computação Científica, IMECC-UNICAMP, Campinas-SP, Brazil}
\email{varao@unicamp.br}

\keywords{Equilibrium States, Partially Hyperbolic Diffeomorphism, Derived from Anosov, Conditional Measures}
\subjclass[2010]{Primary: 37D35; Secondary: 28A50, 37D30}


\begin{abstract}
In this work we address the problem of existence and uniqueness (finiteness) of ergodic equilibrium states for partially hyperbolic diffeomorphisms isotopic to Anosov on $\T^4$, with 2-dimensional center foliation. To do so  we propose to study the disintegration of measures along 1-dimensional subfoliations of the center bundle. Moreover, we obtain a more general result characterizing the disintegration of ergodic measures in our context.
\end{abstract}


\maketitle
\tableofcontents

\section{Introduction}
In dynamics we are interested in the orbits under some law of motion, but it is usually a better way to have a global view of the dynamics to understand it from a probabilistic point of view. That is, instead of being concerned with the orbit of any point we want to say important information about most of the orbits, or at least for those which are (for any reason) particularly important. Hence, it comes to play ergodic theory, which deals with dynamical systems from a probabilistic point of view.\\

From the ergodic theory point of view, invariant measures have play an important role in the theory. In the 1970s, Sinai, Ruelle and Bowen (\cite{Barreira} and references therein) started to study equilibrium states in dynamical systems inspired by techniques and results from statistical mechanics.\\

Given a diffeomorphism $f$ over a compact manifold $M$, an \textit{equilibrium state} for a continuous \textit{potential} $\phi:M\to \mathbb{R}$ is an $f$-invariant Borel probability measure $\mu$ that maximizes the quantity $h_{\mu}(f)$+$\int\phi d\mu$ among all $f$-invariant measures. When $\phi\equiv 0,$ the measure $\mu$ is called \textit{maximal entropy measure}. An interesting problem of ergodic theory is to determinate existence and uniqueness of equilibrium states with respect to some class of potentials. In the setting of uniformly hyperbolic diffeomorphisms, Bowen \cite{Bowen1} solved this problem for Hölder continuous potentials, and in other contexts there exists many approach for this problem. \\

Since the pioneering works of Sinai, Ruelle and Bowen on thermodynamical formalism many results have been obtained, in particular in the context of non-uniformly hyperbolic maps and partially hyperbolic diffeomorphisms. Climenhaga and Thompson extended Bowen's techniques for a nonuniform setting \cite{ClimenhagaFisherThompson,ClimenhagaFisherThompson1}, using those general techniques, Climenhaga \emph{et. al.} \cite{ClimenhagaFisherThompson1, ClimenhagaFisherThompson} proved that robustly transitive diffeomorphisms introduced by Mañé and Bonatti-Viana have a unique equilibrium state for natural classes of potentials. For potentials with a small variational condition, Rios and Siqueira \cite{RiosSiqueira} obtained uniqueness of equilibrium states for partially hyperbolic horseshoes. For the same class of potentials that we will study in this work, Carvalho and Pérez \cite{CarvalhoPerez} obtained similar results about equilibrium states for a class of skew product. Recently, Climenhaga \emph{et. al.} \cite{ClimenhagaPesinZelerowicz} studied uniqueness of equilibrium states for certain transitive partially hyperbolic diffeomorphisms.\\ 

In this work we are also concerned with thermodynamical formalism in the less explored context of a partially hyperbolic diffeomorphism with higher dimensional center foliaiton (i.e. two dimensional or higher) which are isotopic to Anosov. Using disintegration techniques of measures along the center foliation, under some extra assumptions, we obtain similar results to those Crisostomo and Tahzibi \cite{CrisostomoTahzibi}. We give conditions on $f$ and its center foliation under which one can control the geometry of the preimage of the semiconjugacy. Roughly speaking, our conditions are
\begin{itemize}
\item that $f$ is dynamically coherent and the center foliation splits into line-bundles, and
\item  that the preimage of a point $x$ under the semiconjugacy is contained in a unique center leaf and has a ``controlled geometry'', namely a rectangle-like structure.
\end{itemize}
In particular these assumptions are satisfied by the maps considered by Buzzi \emph{et. al.} \cite{BuzziFisherSambarinoVasquez} and Carrasco \emph{et. al.} \cite{CarrascoLizanaPujalsVasquez}, such maps are isotopic to a linear Anosov. \\

Under these hypotheses we address the problem of existence and uniqueness 
(or finiteness) of the equilibrium states. On what follows the set $C$ is the set where the semiconjugacy of $f$ with its linearization is not injective. We now present our results:

\begin{thmx}\label{th.Equilibrium}
Let $f:\T^4\to \T^4$ be a partially hyperbolic diffeomorphism isotopic to Anosov. Under the above assumptions if the measure of $C$ is zero we have uniqueness of the equilibrium state. Otherwise, the disintegration of the measure on the center foliation is atomic and, under some additional assumptions, the equilibrium state is ``virtually hyperbolic'' and not unique.
\end{thmx}

By virtually hyperbolicity we mean that there exists a full measurable invariant subset which intersects each center leaf in at most one point. This result would actually follow from a more general one regarding ergodic measures:

\begin{thmx}\label{th.Ergodic}
Let $f:\T^4\to \T^4$ be a partially hyperbolic diffeomorphism isotopic to Anosov. Under the above assumptions if the measure of $C$ is zero the system is almost conjugated to an Anosov. Otherwise, the disintegration of the measure is atomic and, under some additional assumptions, the measure is ``virtually hyperbolic''.
\end{thmx}

A precise statement on the assumptions and the main results are given in Section \ref{sec.Main} int Theorem \ref{main:equilibrium} and Theorem \ref{main:ergodic} respectively. The novelty with respect to previous works is the study for partially hyperbolic diffeomorphims with 2-dimensional (or higher) center foliations. This is done by a careful study on the disintegration of the measures along the line-bundles of the center foliation. \\

\begin{remark}
We remark that the results of theorems \ref{th.Equilibrium} and \ref{th.Ergodic}  are valid for partially hyperbolic diffeomorphisms $f : \T^n \to \T^n$ isotopic to Anosov with $k$-dimensional center bundle with $1\leq k <n$. Provided that they satisfy the assumptions above. The proof for the higher dimensional case follows in a similar way as in the 2-dimensional case. A further discussion about it will be presented in Section \ref{sec.ProofMain}.
\end{remark}

The remainder of the article is organized as follows. In the next Section we
discuss some necessary preliminaries in equilibrium states, partially hyperbolic dynamics and disintegration of the measures. In Section \ref{sec.Main} we give precise statements os our main results, while their proofs are presented in sections \ref{sec.ProofMain}, \ref{sec.ProofCor} and \ref{sec.ProofEq}.

\section{Preliminaries}

\subsection{Entropy and equilibrium states}
Let $(M,d)$ be a compact metric space and $f:M\to M$ a continuous map. For $\delta\in (0,1)$, $n\in \mathbb{N}$ and $\epsilon>0$, a finite set $E\subset M$ is called an $(n,\epsilon,\delta)$-\textit{covering} if the union of the all $\epsilon$-balls, $B_{n}(x,\epsilon)=\{y\in M: d(f^{i}(x),f^{i}(y))<\epsilon\}$, centered at points $x\in E$ has $\mu$-measure greater than $\delta$. The \textit{metric entropy} is defined by $$h_{\mu}(f)=\displaystyle\lim_{\epsilon \rightarrow 0}\displaystyle\limsup_{n\rightarrow \infty}\frac{1}{n}\log \min\{\# E: E\subseteq M \ {\rm is  \ a \ }(n,\epsilon,\delta)\mbox{\rm{-covering}\ set} \}.$$
A set $E\subseteq M$ is said to be $(n,\epsilon)$-\textit{separated}, if for every $x, y \in E, x \neq y$, there exists $i\in\{0,\ldots, n-1\}$ such that $d(f^{i}x, f^{i}y)\geq \epsilon$. The \textit{topological entropy} on a non-empty compact set $K\subset M$ is defined by $$h(f,K)=\displaystyle\lim_{\epsilon \rightarrow 0}\displaystyle\limsup_{n\rightarrow \infty}\frac{1}{n}\log \sup\{\# E: E\subseteq K \ \mbox{\rm{is}} \ (n,\epsilon)\mbox{\rm{-separated}} \}.$$ We denote $h_{top}(f):=h(f,M).$

\begin{definition}
Let $f:M\rightarrow M$ be a continuous map over a compact manifold $M$. An $f$-invariant Borel probability measure $\mu$ is an \textbf{equilibrium state} for $f$ with respect to a potential $\phi\in C^{0}(M,\mathbb{R})$ if it satisfies $$h_{\mu}(f)+\displaystyle\int \phi d\mu= \sup\{h_{\nu}(f)+\displaystyle\int \phi d\nu:\nu\in \mathcal{M}(f)\},$$ where $h_{\mu}(f)$ is the metric entropy of $f$  with respect to $\mu$. If $\phi\equiv 0$, $\mu$ is called \textbf{measure of maximal entropy}.
\end{definition}

\subsection{Partially hyperbolicity}
Let $f:M\to M$ be a diffeomorphism defined on a compact manifold $M$, $f$ is said to be \textit{partially hyperbolic} if:
\begin{enumerate}
 \item There exists a non-trivial splitting of the tangent bundle $TM=E^s\oplus E^c\oplus E^u$ invariant under the derivative $Df$;
 \item There exist a Riemannian metric $\|\cdot\|$ on $M$, such that we have positive continuous functions $\nu$, $\hat{\nu}$, $\gamma$, $\hat{\gamma}$ with $\nu$, $\hat{\nu}<1$ and $\nu<\gamma<{\hat\gamma}^{-1}<{\hat{\nu}}^{-1}$ such that, for any unit vector $v\in T_xM$,
\begin{alignat*}{2}
& \|Df(x)v \| < \nu(x) & \quad & \text{if } v\in E^s(x),
 \\
\gamma(x) < & \|Df(x)v \|  < {\hat{\gamma}(x)}^{-1} & & \text{if } v\in E^c(x),
 \\
{\hat{\nu}(x)}^{-1} < & \|Df(x)v\| & &  \text{if } v\in E^u(x).
\end{alignat*}
\end{enumerate}
The bundles $E^s$, $E^u$, $E^c$ are called the stable, unstable and center bundle respectively. It is well-known that the stable and unstable bundles integrate to $f$-invariant foliations $\F^s$ and $\F^u$ \cite{HirschPughShub}. The leaf of $\F^{\sigma}$ containing $x$ will be called $W^{\sigma}(x)$, for $\sigma= s,u$. Such foliations are $f$-invariant, that means $f$ sends leaves to leaves. 

\begin{remark}
Not always the central bundle $E^{c}$ may be tangent to an invariant foliation, but whenever such a foliation exists, it is denoted by $\mathcal{F}^{c}.$
\end{remark}

\begin{definition}
A partially hyperbolic diffeomorphism $f:M\rightarrow M$ is called dynamically coherent if there exist invariant foliations $\mathcal{F}^{c \sigma}$ tangent to $E^{c \sigma}=E^{c}\oplus E^{\sigma}$ for $\sigma=s,u.$
\end{definition}

\begin{remark}
In this article we assume that $f$ preserves the orientation on the $\F^i$ leaves.
\end{remark}

\subsection{Derived from Anosov}
\begin{definition}
A $C^1$-diffeomorphism $f:\T^d\to \T^d$ is called Derived from Anosov (DA) if it is isotopic to its action in the homology $A: H_1(\T^d )\to H_1(\T^d )$. We call $A$ the linear part of $f$.
\end{definition}

By a well-known result of Franks \cite{Franks} there exist a semiconjugacy $H:\T^d \to \T^d$ between $f$ and $A$, that is, $H\circ f=A\circ H.$ Moreover, its lift $\tH$ to $\R^d$ semiconjugates $\tf$ to $\tilde{A}$, and for some constant $K$ we have
\[
 \|\tH-\id\|_{C^0}\leq K.
\]
\vspace{0.1cm}

In particular, $\tH$ is proper. The constant $K$ depends continuously on $f$, and tends to zero as $f$ tends to $A$ in the $C^1$ norm. For every $\tx\in\R^d$, each $\tH^{-1}(\tx)$ is a compact set whose diameter is uniformly bounded from above $\diam (\tH^{-1}(\tx))\leq 2E$.\\

\begin{remark}
Take $\mu$ any $f$-invariant measure and let $\nu=H_{\ast}\mu$. It is well-known that $h_{\mu}(f)\geq h_{\nu}(A)$. Furthermore, the Ledrappier-Walters Variational principle \cite{LedrappierWalters} says that
\begin{equation}\label{eqn:LedWalt}
 \sup_{\mu:H_{\ast}\mu=\nu}h_{\mu}(f)=h_{\nu}(A)+\int_{\T^d}h(f,H^{-1}(x))d\nu(x).
\end{equation}
Hence, when $h(f,H^{-1}(x))=0$ for every $x\in\T^d$, we have that for any $f$-invariant measure $\mu$
\begin{equation}\label{eqn:EqualityMetricEntropy}
 h_{\mu}(f)= h_{\nu}(A).
\end{equation}
\end{remark}

Consider a potential $\phi:\T^d\to\R$ for $A$ and define $\varphi=\phi\circ H$. A well-known result due to Bowen \cite{Bowen} states that if $\mu$ is an equilibrium state of $(f,\varphi)$ then, $\nu=H_{\ast}\mu$ is an equilibrium state of $(A,\phi)$. Furthermore, the reciprocate is also true under certain conditions:\\

\begin{lemma}\label{lemma:ExistenceMu}
Let $f:\mathbb{T}^d \to \mathbb{T}^d$ be a DA partially hyperbolic diffeomorphism. Let $\phi:\T^d\to\R$ be a continuous potential and define $\varphi=\phi\circ H$ . Assume that $h(f,H^{-1}(x))=0$ for every $x\in \mathbb{T}^{d}$. If $\nu$ is an equilibrium state for $(A,\phi)$, then every $\mu\in \mathcal{M}(f)$ such that $H_{\ast}\mu=\nu$ is an equilibrium state for $(f,\varphi=\phi\circ H)$.
\end{lemma}

\begin{proof}
Let $\nu$ be an equilibrium state for $(A,\phi)$. By the Riezs theorem and the compactness of the set of Borel probability measures on $\T^d$ we can guarantee the existence of an $f$-invariant measure $\mu$ such that $\nu=H_{\ast}\mu$ (see for example \cite[Lemma~4.3]{Bowen} for a similar construction). Moreover, by (\ref{eqn:EqualityMetricEntropy}) we have that
\begin{align*}
\sup \left\lbrace h_{\eta}(f)+\int \varphi d\eta: \eta\in \mathcal{M}(f) \right\rbrace 
&=\sup \left\lbrace h_{H_{\ast}\eta}(A)+\int \phi d H_{\ast}\eta: \eta\in \mathcal{M}(f) \right\rbrace,\\
&\leq \sup \left\lbrace h_{\hat{\nu}}(A)+\int \phi d\hat{\nu} : \hat{\nu}\in \mathcal{M}(A) \right\rbrace,\\
&\leq h_{\nu}(A)+\int \phi d\nu.
\end{align*}
Therefore, any $f$-invariant measure $\mu$ satisfying that $\nu=H_{\ast}\mu$ is an equilibrium state for $(f,\varphi)$.
\end{proof}  

\subsection{Rectangle structure}\label{sec.Rect}

We say that a DA partially hyperbolic diffeomorphism $f:\T^d\to\T^d$ has \textbf{rectangle structure} in the center bundle if:

\begin{enumerate}[label=\Alph*.]
\item \label{Hyp.R1} There exist a splitting $E^c=E^1\oplus E^2$ where each $E^i$ is a line-bundle and integrates to an $f$-invariant foliation $F^i$, for $i=1,2$.
\item \label{Hyp.R2} For every $x\in\T^d$, if $z,z'\in H^{-1}(x)$ and $z'\in \F^i(z)$ for some $1\leq i\leq 2$, then 
\[
 [z,z']_i\subset H^{-1}(x),
\]
where $[z,z']_i$ is the closed interval inside $\F^i(z)$ with end points $z$ and $z'$.
\item \label{Hyp.R3} For each $x\in \T^d,$ $H^{-1}(x)$ is a finite union of rectangles contained in a unique center leaf of $\mathcal{F}^c$.
\end{enumerate}

The \emph{rectangles} mentioned above are compact sets obtained in the following inductive procedure. Let $z_0,...,z_k$, with $1\leq k\leq \ell$, be points in $H^{-1}(x)$ such that $z_j\in \F^{i_j}(z_0)$. We construct the rectangle (of dimension $k$ and corner $z_0$) by starting with $R_1=[z_0,z_1]_{i_1}\subset \F^{i_1}(z_0)$. Taking $i_2\neq i_1$ we can define $R_2$ as the trace inside $\F^c(z_0)$ of the set obtained by sliding $R_1$ along $[z_0, z_2]_{i_2}\subset \F^{i_2}(z_0)$, that is,
\[
 R_2=\bigcup_{w\in[z_0, z_2]_{i_2}}[w,y(w)]_{i_1},
\]
where $[w,y(w)]_{i_1}$ is the image of $[z_0,z_1]_{i_1}$ by the $\F^{i_2}$-holonomy. Continuing this way, we can define $R_k$ as
\[
 R_k=\bigcup_{w\in[z_0, z_{k}]_{i_2}}R^{k-1}(w),
\] 
where $R^{k-1}(w)$ is a rectangle of dimension $k-1$ and corners $z_0,...,z_{k-1}$  obtained as the image of $R_{k-1}$ in the corresponding center manifold by the $\F^{i_k}$-holonomy sending $z_0$ in $w$.\\

In a recent work, Carrasco \emph{et. al.} \cite{CarrascoLizanaPujalsVasquez} proved that, if the center bundle of $f$ is strongly simple (it decomposes into one dimensional sub-bundles with global product structure), then it has rectangle structure in the center bundle. Moreover, they also proved that for every $x\in\T^d$ then:
\[
 h(f,H^{-1}(x))=0.
\]

\subsection{Disintegration of measures}
Let $(M,\B, \mu)$ be a probability space, where $M$ is a compact metric space, $\B$ the borelian $\sigma$-algebra and $\mu$ a probability measure. Let $\P$ be a partition of $M$ and $\hat{\B}=\pi_{\ast}\B$, $\hat{\mu}=\pi_{\ast}\mu$ where $\pi:M\rightarrow M/\P$ is the canonical projection that assigns to each point $x\in M$ the element $\P(x)$ of the partition that contains it, then $(\tilde{M}:=M/\P,\hat{\B},\hat{\mu})$ is a probability space.

\begin{definition}
A disintegration of $\mu$ with respect to $\mathcal{P}$ is a family $\{\mu_{P}\}_{P\in \mathcal{P}}$ of conditional probability measures on $M$ such that:
\begin{enumerate}
\item given $\phi\in C^{0}(M),$ then $P\mapsto \int \phi d\mu_{P}$ is measurable;
\item $\mu_{P}(P)=1, \hat{\mu}$-a.e.;
\item $\mu=\int_{\tilde{M}}\mu_{P}d\hat{\mu}$, i.e, if $\phi\in C^{0}(M),$ then $\int\phi d\mu=\int_{\tilde{M}}\int_{P}d\mu_{P}d\tilde{\mu}.$
\end{enumerate}\vspace{0.1cm}

When it is clear which partition we are referring to, we say that the family $\{\mu_{P}\}$ disintegrates the measure $\mu$. There exists an equivalent form of writing the disintegration formula
above: $$\mu=\int_{M}\mu_{x}d\mu$$ by considering the conditional measures $\mu_{x}, x\in M$ where $\mu_{x}=\mu_{y}$ if $y\in \mathcal{P}(x).$
\end{definition}

\begin{proposition}[\cite{OliveiraVianaErgodic}]
If $\{\mu_{P}\}_{P\in \mathcal{P}}$ and $\{\tilde{\mu}_{P}\}_{P\in \mathcal{P}}$ are disintegrations of $\mu$ with respect to $\mathcal{P}$, then $\mu_{P}=\tilde{\mu}_{P}$ for $\hat{\mu}$-almost every $P\in\mathcal{P}$.
\end{proposition}

The previous proposition asserts that disintegrations are essentially unique, when they exist. Consequently, for an invariant measure it follows that:
\begin{corollary}
If $f:M\to M$ preserves a probability measure $\mu$ and the partition $\mathcal{P}$, then $f_{\ast}\mu_{P}=\mu_{f(P)}$ $\hat{\mu}$-a.e.
\end{corollary}

\begin{definition}
We say that a partition $\P$ of $M$ is measurable with respect to probability measure $\mu$ if there exist a measurable family $\{A_{i}\}_{i\in \mathbb{N}}$ and a measurable set $C$ of full measure such that if $B\in \P$, then there exists a sequence $\{B_{i}\}_{i\in \N}$, where $B_{i}\in \{A_{i}, A_{i}^{c}\}$ such that $B\cap C=\bigcap_{i\in \N}B_{i}\cap C.$
\end{definition}

The next theorem guarantees the existence of disintegrations with respect to a measurable partition.

\begin{theorem}[Rokhlin's Disintegration \cite{Rokhlin}]\label{tro}
Let $\mathcal{P}$ be a measurable partition of a compact metric space $M$ and $\mu$ a Borel probability measure. Then, $\mu$ admits some disintegration with respect to $\mathcal{P}$.
\end{theorem}

The partition by leaves of a foliation may be non-measurable in general. For instance, this is the case for the stable and unstable foliations of Anosov diffeomorphisms with respect to measures of non vanishing metric entropy. Instead, one must consider disintegrations on compact foliated boxes. This conditional measures will depend on the foliated boxes, however, in \cite[Lemma~ 3.2]{AvilaVianaWilkinson} they proved that this measures are defined up to scaling. That is, they are equivalence classes where one identifies any two (possibly infinite) measures that differ only by a constant factor. 

\begin{definition}
We say that a foliation $\F$ has atomic disintegration with respect to a measure $\mu$ if the conditional measures on any foliated box are sum of Dirac measures.
\end{definition}

Another way to define atomic disintegration is as follows: there exist a full measurable subset $Z$ such that $Z$ intersects all leaves in at most a countable set.\\

Even though the disintegration of a measure along a general foliation is defined in compact foliated boxes, it makes sense to say that a foliation $\F$ has a quantity $k\in\N$ of atoms per leaf. The meaning of ``per leaf'' should always be understood as a generic leaf, i.e. almost every leaf. That means that there is a set $A$ of $\mu$-full measure which intersects a generic leaf on exactly $k$ points.

\begin{definition}
We say that a measure is virtually hyperbolic if there is a full measure set which intersects the center leaf in at most one point.
\end{definition}

Let us finish this section by enunciating a well-known result known as the Measurable Choice Theorem. This result has appeared in the context of Decision Theory from Economics and as been proved by R. J. Aumann \cite{Aumann}:

\begin{theorem}[Measurable Choice Theorem]\label{th.Measurable}
Let $(T,\mu)$ be a $\sigma$-finite measure space, let $S$ be a Lebesgue space, and let $G$ be a measurable subset of $T\times S$ whose projection on $T$ is all of $T$. Then there is a measurable function $g:T\to S$, such that $(t, g(t))\in G$ for almost all $t\in T$.
\end{theorem}
\section{Main results}\label{sec.Main}
For now on we focus in DA partially diffeomorphisms in $\T^4$ with 2-dimensional center bundle. Consider the set 
\begin{equation}\label{eqn:DefC}
 C:=\{x\in \T^4:\#H^{-1}H(x)>1\}.
\end{equation}

We claim that $C$ is a measurable set. One may check that by first observing that by simply reproducing \textit{ipsis litteris} the proof of \cite[Lemma 3.2]{PonceTahzibiVaraoBernoulli} only changing $\T^3$ by $\T^4$ and $\R^3$ by $\R^4$ one obtains that the set $H(C)$ is a measurable set. Hence $C=H^{-1} H(C)$ is a measurable set. Moreover, notice that $C$ is $f$-invariant. \\

We are now able to properly state Theorem \ref{th.Ergodic}:

\begin{theorem}\label{main:ergodic}
Let $f:\mathbb{T}^{4}\rightarrow \mathbb{T}^{4}$ be a DA partially hyperbolic diffeomorphism dynamically coherent and with rectangle structure in the center bundle. Assume that $f$ preserves the orientation of $\F^{i}$, for  $i=1,2$. Then, if $\mu$ be an ergodic probability for $f$
\begin{enumerate}
\item If $\mu(C)=0$, $(f,\mu)$ is almost conjugate to an Anosov.
\item If $\mu(C)=1$, $C$ defines a partition such that $\mu$ has atomic disintegration with a finite number of atoms. 
\end{enumerate}
\end{theorem}

Mixed derived from Anosov examples $g:\mathbb{T}^d\to \mathbb{T}^d$ introduced by Buzzi \emph{et. al.} \cite[Section~5]{BuzziFisherSambarinoVasquez} satisfy the dynamical coherent and rectangle structure hypothesis. In particular, the center foliation $\mathcal{F}^c$ admits two invariant 1-dimensional sub-foliations $\mathcal{F}^{cu}, \mathcal{F}^{cs}$ such that $H^{-1}(x)\cap \mathcal{F}^{cu}_{loc}(x)$ and $H^{-1}(x)\cap \mathcal{F}^{cs}_{loc}(x)$ are segments in the center foliation. Another class of derived from Anosov $f:\T^4\to \T^{4}$ that satisfy these assumptions was studied by Carrasco \emph{et. al.} \cite[Theorem~A and Section~3]{CarrascoLizanaPujalsVasquez}.

The virtually hyperbolicity result mention in Theorem \ref{th.Ergodic} follows from the next corollary.

\begin{corollary}\label{cor.main}
Under the assumptions of Theorem \ref{main:ergodic}. Let us assume that $\nu:=H_{\ast}\mu$ has full support and the semiconjugacy $H$ sends center leaves of $f$ to center leaves of $A$. If one of the following conditions is satisfied
\begin{enumerate}
     \item \label{it.2} The center direction of $A$ is expansive or contractive.
     \item \label{it.3} $H(\mathcal \F^i)$ is some invariant foliation of $A$, for each $i=1,2$. 
\end{enumerate}
Then, $\mu$ is virtually hyperbolic.
\end{corollary}

Now, we consider the following property: 
\begin{enumerate}[label=(H)]
\item \label{Hyp.5} $h(f, H^{-1}(x))=0$, for all $x\in \mathbb{T}^d$.
\end{enumerate}

As mention before in Section \ref{sec.Rect}, \cite[Theorem~4.1]{CarrascoLizanaPujalsVasquez} guarantees that DA partially hyperbolic diffeomorphisms with central bundle $E^{c}$ strongly simple (see definition \cite[Definition~1.4]{CarrascoLizanaPujalsVasquez}) also satisfy the assumption \ref{Hyp.5}. Furthermore, \cite[Corollary~5.2]{BuzziFisherSambarinoVasquez} guarantees that mixed derived from Anosov examples satisfy the assumption \mbox{\rm\ref{Hyp.5}}.\\

Let us recall that if $f:\T^4\to \T^4$ is a DA partially hyperbolic diffeomorphism with a dominated splitting, then the existence of equilibrium states associated to any continuous potential is guaranteed as a consequence of the work of Díaz \emph{et. al.} \cite[Corollary~1.3]{DiazFisherPacificoVieitez}. The assumption \ref{Hyp.5} will guarantee the existence of the equilibrium state as a consequence of Lemma \ref{lemma:ExistenceMu}. Moreover, Theorem \ref{th.Equilibrium} also gives a partial answer to the uniqueness problem of the equilibrium states. We now proceed to formaly state Theorem \ref{th.Equilibrium}:

\begin{theorem}\label{main:equilibrium}
Let $f:\mathbb{T}^{4}\rightarrow \mathbb{T}^{4}$ be a DA partially hyperbolic diffeomorphism dynamically coherent, with rectangle structure in the center bundle and satisfying \ref{Hyp.5}. Let us assume that $\phi$ is a continuous potential such that $(A,\phi)$ that has a unique equilibrium state with full support. Define the potential $\varphi=\phi\circ H$ and $\mu$ be any ergodic equilibrium state of $f$ with respect to $\varphi$. If $f$ preserves the orientation of $\F^{i}, \ i=1,2$ then:
\begin{enumerate}
    \item if $\mu(C)=0$, then $\mu$ is the unique equilibrium state;
    \item if $\mu(C)=1$, $C$ defines a partition such that $\mu$ has atomic disintegration with a finite number of atoms. Moreover, if the semiconjugacy $H$ sends center leaves of $f$ to center leaves of $A$ and one of the following conditions is satisfied
    \begin{enumerate}
     \item \label{it.Eq2} The center direction of $A$ is expansive or contractive.
     \item \label{it.Eq3} $H(\mathcal \F^i)$ is some invariant foliation of $A$, for each $i=1,2$. 
    \end{enumerate}
Then $\mu$ is virtually hyperbolic and it is not a unique equilibrium state.
\end{enumerate}
\end{theorem}

\section{Proof of Theorem \ref{th.Ergodic}}\label{sec.ProofMain}

From now on we assume $\mu(C)=1$ and let us prove that the partition determined by $C$ has atomic disintegration. That is, consider the partition:
\[
 \P:=\{\P(x):=H^{-1}H(x)|x\in C\}.\\
\] 

Let us prove that $\P$ is a measurable partition with respect to any measure considered. Let $\{A_i\}_{i \in \N}$ be a countable basis for the topology of $\T^4$. Now for any point $x \in \T^4$ we have sets $B_i \equiv B_i(x) \in \{A_i, A^c_i\}$ such that $\{ x \} = \cap_{i \in \N} B_i$. Since $\{H^{-1}(A_i)\}$ is a measurable set (because $A_i$ is an open set and $H$ is continuous) notice that 
\[
 H^{-1}(x) = \bigcap_{i \in \N}H^{-1}(B_i).
\]
Thus proving that $\P$ is a measurable partition. Moreover, it is easy to see that $\P$ is left invariant by $f$, that is $f(\P(x))=P(f(x))$.\\

Assume, without loss of generality, that $\F^1$ is oriented and $f$ preserves its orientation. We define another partition $\Q$ as the one whose elements are the connected components of the intersection of elements of $\P$ and $\F^1$. That is
\[
 \Q:=\{Q(x)=\F^1(x)\cap \P(x)|x\in C\}.\\
\]

Recall that, by the rectangle structure, $H^{-1}(z)$ is a finite union of rectangles in $\F^c$, so we can write for each $x\in C$
\begin{equation}\label{eqn.UnionRectangles}
 \P(x)=\bigcup_{j=1}^{n_x}R_j(x),
\end{equation}
where $n_x$ represents the number of rectangles in the class and $R_j(x)$ denotes a rectangle of dimension $1\leq k_j=k_j(x)\leq 2$ with corners $z_0,...,z_{k_j}$. Moreover, assumption \ref{Hyp.R2} guarantees that $Q(x)$ has only one connected component, an interval or a point. Therefore, the foliation of each element of $\P$ by $\F^1$ is a foliation by compact leaves. Thus, we can consider $\Q$ as a measurable partition. Indeed, any foliation with compact leaves can be consider as a measurable partition, see \cite[Proposition~ 3.7]{AvilaVianaWilkinson2}. Let us denote the conditional measures on $\Q$ by $\mu_x$.\\

\begin{figure}[t!]
\centering
\captionsetup{justification=centering}

\begin{minipage}[c]{.7\linewidth}
\centering

\newrgbcolor{qqttzz}{0 0.2 0.6}
\newrgbcolor{ccqqqq}{0.8 0 0}

\psset{xunit=0.5cm,yunit=0.5cm,algebraic=true,dotstyle=o,dotsize=3pt 0,linewidth=0.8pt,arrowsize=3pt 2,arrowinset=0.25}

\begin{pspicture*}(-3.5,-8)(19.56,7.23)

\parametricplot[linewidth=1.2pt,linecolor=ccqqqq]{-0.2406942442470834}{0.24560912946714195}{1*16.73*cos(t)+0*16.73*sin(t)+-11.22|0*16.73*cos(t)+1*16.73*sin(t)+0.4}
\parametricplot[linestyle=dotted,linecolor=ccqqqq]{-0.38249366103447713}{0.3774431667194532}{1*16.57*cos(t)+0*16.57*sin(t)+-11.06|0*16.57*cos(t)+1*16.57*sin(t)+0.36}
\parametricplot[linestyle=dotted]{1.4906069730008062}{1.6614294498905402}{1*47.58*cos(t)+0*47.58*sin(t)+5.78|0*47.58*cos(t)+1*47.58*sin(t)+-47.57}

\rput[tl](5.05,6.44){\ccqqqq{$ \mathcal{F}^1(z) $}}
\rput[tl](1.53,-0.6){$ \mathcal{F}^2(z) $}

\psbrace(6.03,-3.6)(6.01,4.5){$\mathcal{Q}(z)$}

\begin{scriptsize}
\psdots[dotstyle=*,linecolor=ccqqqq](5.03,-3.59)
\rput[bl](4.1,-3.6){\ccqqqq{$z_0$}}
\psdots[dotstyle=*,linecolor=ccqqqq](5.51,0.02)
\rput[bl](4.8,0.3){\ccqqqq{$z$}}
\psdots[dotstyle=*,linecolor=ccqqqq](5.01,4.47)
\rput[bl](4.22,4.3){\ccqqqq{$z_1$}}
\end{scriptsize}

\end{pspicture*}

\vspace{-1cm}
\subcaption{$R_j$ is a rectangle of dimension 1 contained in a $\F^1$-leaf}\label{fig.RecDim1F1}
\end{minipage}

\vspace{-2cm}
\begin{minipage}[c]{.7\linewidth}
\centering

\newrgbcolor{qqttzz}{0 0.2 0.6}
\newrgbcolor{ccqqqq}{0.8 0 0}

\psset{xunit=0.5cm,yunit=0.5cm,algebraic=true,dotstyle=o,dotsize=3pt 0,linewidth=0.8pt,arrowsize=3pt 2,arrowinset=0.25}

\begin{pspicture*}(-3.5,-8)(21.85,11.37)

\parametricplot[linewidth=1.2pt,linestyle=dotted,linecolor=qqttzz]{-0.2406942442470834}{0.24560912946714195}{1*16.73*cos(t)+0*16.73*sin(t)+-11.22|0*16.73*cos(t)+1*16.73*sin(t)+0.4}
\parametricplot[linestyle=dotted]{1.4592798033166172}{1.7051363958390253}{1*60.59*cos(t)+0*60.59*sin(t)+6.57|0*60.59*cos(t)+1*60.59*sin(t)+-60.06}
\parametricplot[linewidth=1.2pt]{1.4758004302769274}{1.6799377399158195}{1*62.95*cos(t)+0*62.95*sin(t)+6.48|0*62.95*cos(t)+1*62.95*sin(t)+-62.41}
\psline[linewidth=0.4pt,linecolor=ccqqqq]{->}(5.51,0.53)(6.43,1.43)

\rput[tl](5.62,5.19){\qqttzz{$ \mathcal{F}^1(z) $}}
\rput[tl](1.53,0.25){$ \mathcal{F}^2(z) $}

\begin{scriptsize}
\psdots[dotstyle=*,linecolor=ccqqqq](5.51,0.53)
\rput[bl](4.8,1){\ccqqqq{$z$}}
\rput[bl](6.43,1.37){\ccqqqq{$\mathcal{Q}(z)$}}
\psdots[dotstyle=*](-0.38,0.16)
\rput[bl](-0.69,-0.64){$z_0$}
\psdots[dotstyle=*](12.45,0.25)
\rput[bl](12.3,-0.67){$z_1$}
\end{scriptsize}

\end{pspicture*}

\vspace{-1.5cm}
\subcaption{$R_j$ is a rectangle of dimension 1 contained in a $\F^2$-leaf}\label{fig.RecDim1F2}
\end{minipage}

\vspace{0.5cm}
\begin{minipage}[c]{.7\linewidth}
\centering

\newrgbcolor{qqttzz}{0 0.2 0.6}
\newrgbcolor{ccqqqq}{0.8 0 0}

\psset{xunit=0.55cm,yunit=0.55cm,algebraic=true,dotstyle=o,dotsize=3pt 0,linewidth=0.8pt,arrowsize=3pt 2,arrowinset=0.25}

\begin{pspicture*}(-3.5,-8)(19.56,7.23)

\parametricplot[linewidth=1.2pt]{1.4658919836717454}{1.675700669918048}{1*76.4*cos(t)+0*76.4*sin(t)+6|0*76.4*cos(t)+1*76.4*sin(t)+-79.98}
\parametricplot[linewidth=1.2pt,linecolor=qqttzz]{-0.23260396708662157}{0.23260396708662137}{1*17.35*cos(t)+0*17.35*sin(t)+-18.89|0*17.35*cos(t)+1*17.35*sin(t)+0}
\parametricplot[linewidth=1.2pt,linecolor=qqttzz]{-0.24231811194623631}{0.24231811194623612}{1*16.67*cos(t)+0*16.67*sin(t)+-2.18|0*16.67*cos(t)+1*16.67*sin(t)+0}
\parametricplot[linewidth=1.2pt]{1.4517784610528233}{1.68981419253697}{1*67.38*cos(t)+0*67.38*sin(t)+6|0*67.38*cos(t)+1*67.38*sin(t)+-62.9}
\parametricplot[linewidth=0.4pt,linestyle=dotted]{1.438933038681401}{1.7070124324631315}{1*59.92*cos(t)+0*59.92*sin(t)+6.34|0*59.92*cos(t)+1*59.92*sin(t)+-62.39}
\parametricplot[linewidth=0.4pt,linestyle=dotted]{1.4366131522745456}{1.7048701232109778}{1*59.88*cos(t)+0*59.88*sin(t)+6.35|0*59.88*cos(t)+1*59.88*sin(t)+-61.35}
\parametricplot[linewidth=0.4pt,linestyle=dotted]{1.4545880574749945}{1.6891027161937202}{1*68.47*cos(t)+0*68.47*sin(t)+6.52|0*68.47*cos(t)+1*68.47*sin(t)+-69}
\parametricplot[linewidth=0.4pt,linestyle=dotted]{1.4362722835818533}{1.7053205450573163}{1*59.72*cos(t)+0*59.72*sin(t)+6.48|0*59.72*cos(t)+1*59.72*sin(t)+-59.19}
\parametricplot[linewidth=0.4pt,linestyle=dotted]{1.4398908084292725}{1.7045415640600527}{1*60.7*cos(t)+0*60.7*sin(t)+6.53|0*60.7*cos(t)+1*60.7*sin(t)+-59.19}
\parametricplot[linewidth=0.4pt,linestyle=dotted]{1.4369775988146314}{1.7047383570510684}{1*59.99*cos(t)+0*59.99*sin(t)+6.36|0*59.99*cos(t)+1*59.99*sin(t)+-57.46}
\parametricplot[linewidth=0.4pt,linestyle=dotted]{1.4275639066583594}{1.7112163309832855}{1*56.64*cos(t)+0*56.64*sin(t)+6.13|0*56.64*cos(t)+1*56.64*sin(t)+-53.06}
\parametricplot[linewidth=0.4pt,linestyle=dotted,linecolor=qqttzz]{-0.2589490416790463}{0.2561277029838513}{1*15.73*cos(t)+0*15.73*sin(t)+-16.22|0*15.73*cos(t)+1*15.73*sin(t)+0.13}
\parametricplot[linewidth=0.4pt,linestyle=dotted,linecolor=qqttzz]{-0.24847858790406097}{0.2533679762640946}{1*16.16*cos(t)+0*16.16*sin(t)+-15.63|0*16.16*cos(t)+1*16.16*sin(t)+0.16}
\parametricplot[linewidth=0.4pt,linestyle=dotted,linecolor=qqttzz]{-0.21560038250441327}{0.21120664301181347}{1*18.97*cos(t)+0*18.97*sin(t)+-17.52|0*18.97*cos(t)+1*18.97*sin(t)+0.32}
\parametricplot[linewidth=0.4pt,linestyle=dotted,linecolor=qqttzz]{-0.22569590480210344}{0.22017112155274524}{1*18.19*cos(t)+0*18.19*sin(t)+-15.73|0*18.19*cos(t)+1*18.19*sin(t)+0.39}
\parametricplot[linewidth=0.4pt,linestyle=dotted,linecolor=qqttzz]{-0.2299891503973468}{0.23557478716500682}{1*17.44*cos(t)+0*17.44*sin(t)+-13.94|0*17.44*cos(t)+1*17.44*sin(t)+0.34}
\parametricplot[linewidth=0.4pt,linestyle=dotted,linecolor=qqttzz]{-0.23475855500593656}{0.22522388125080922}{1*17.67*cos(t)+0*17.67*sin(t)+-13.19|0*17.67*cos(t)+1*17.67*sin(t)+0.5}
\parametricplot[linewidth=1.2pt,linecolor=ccqqqq]{-0.2406942442470834}{0.24560912946714195}{1*16.73*cos(t)+0*16.73*sin(t)+-11.22|0*16.73*cos(t)+1*16.73*sin(t)+0.4}
\parametricplot[linewidth=0.4pt,linestyle=dotted,linecolor=qqttzz]{-0.2368613212533699}{0.23192095946279598}{1*17.34*cos(t)+0*17.34*sin(t)+-10.85|0*17.34*cos(t)+1*17.34*sin(t)+0.49}
\parametricplot[linewidth=0.4pt,linestyle=dotted,linecolor=qqttzz]{-0.25504126325191834}{0.244983345635518}{1*16.28*cos(t)+0*16.28*sin(t)+-8.77|0*16.28*cos(t)+1*16.28*sin(t)+0.52}
\parametricplot[linewidth=0.4pt,linestyle=dotted,linecolor=qqttzz]{-0.2644294392483273}{0.25447439575352954}{1*15.69*cos(t)+0*15.69*sin(t)+-7.15|0*15.69*cos(t)+1*15.69*sin(t)+0.5}
\parametricplot[linewidth=0.4pt,linestyle=dotted,linecolor=qqttzz]{-0.23990094560670006}{0.24518545590870444}{1*16.76*cos(t)+0*16.76*sin(t)+-7.27|0*16.76*cos(t)+1*16.76*sin(t)+0.34}
\parametricplot[linewidth=0.4pt,linestyle=dotted,linecolor=qqttzz]{-0.24954157024380397}{0.2550438078858211}{1*16.11*cos(t)+0*16.11*sin(t)+-5.58|0*16.11*cos(t)+1*16.11*sin(t)+0.29}
\parametricplot[linewidth=0.4pt,linestyle=dotted,linecolor=qqttzz]{-0.27735663855019776}{0.2768195775981685}{1*14.69*cos(t)+0*14.69*sin(t)+-3.14|0*14.69*cos(t)+1*14.69*sin(t)+0.28}
\parametricplot[linewidth=0.4pt,linestyle=dotted,linecolor=qqttzz]{-0.25942436688335313}{0.2651887834136855}{1*15.48*cos(t)+0*15.48*sin(t)+-2.92|0*15.48*cos(t)+1*15.48*sin(t)+0.15}
\parametricplot[linewidth=0.4pt,linestyle=dotted,linecolor=qqttzz]{-0.25925503532936034}{0.25861054882642875}{1*15.65*cos(t)+0*15.65*sin(t)+-2.12|0*15.65*cos(t)+1*15.65*sin(t)+0.11}
\parametricplot[linewidth=0.4pt,linestyle=dotted,linecolor=ccqqqq]{-0.38249366103447713}{0.3774431667194532}{1*16.57*cos(t)+0*16.57*sin(t)+-11.06|0*16.57*cos(t)+1*16.57*sin(t)+0.36}

\rput[tl](5.06,6.56){\ccqqqq{$ \mathcal{F}^1(z) $}}

\psbrace(6.03,-3.4)(6.01,4.3){$\mathcal{Q}(z)$}

\begin{scriptsize}
\psdots[dotstyle=*,linecolor=qqttzz](-2,-4)
\rput[bl](-3.26,-5.11){\qqttzz{$z_0$}}
\psdots[dotstyle=*,linecolor=blue](14,-4)
\rput[bl](14.77,-5.04){\blue{$z_1$}}
\rput[bl](6.3,-5.2){$\mathcal{F}^2$}
\psdots[dotstyle=*,linecolor=qqttzz](-2,4)
\rput[bl](-3.29,4.11){\qqttzz{$z_2$}}
\rput[bl](-3.48,0.14){\qqttzz{$\mathcal{F}^1$}}
\psdots[dotstyle=*,linecolor=ccqqqq](5.51,0.02)
\rput[bl](5.0,-0.2){\ccqqqq{$z$}}
\end{scriptsize}

\end{pspicture*}

\vspace{-0.5cm}
\subcaption{$R_j$ is a rectangle of dimension 2}\label{fig.RecDim2}
\end{minipage}

\caption{Partition $\Q$.}\label{fig.PartQx}
\end{figure}

It is easy to see that the partition $\Q$ is $f$-invariant and, therefore, $f_*\mu_x=\mu_{f(x)}$. Consider $\pi:C\rightarrow \widehat{C}:=C/\Q$ the canonical projection that assigns to each point $x\in C$ the element $\Q(x)$ of the partition that contains it.  Denote the quotient measure as $\hat{\mu}=\pi_{\ast}\mu$. \\


\begin{lemma}\label{lem.AtomicDis}
The measure $\mu$ has atomic disintegration with respect to the partition $\Q$.
\end{lemma}
\begin{proof}
We want to show that the conditional measure $\mu_x$ is a countable linear combination of Dirac masses for $\hat{\mu}$-almost every $Q(x)\in \Q$. By absurd assume there exist a set $\hat{\Lambda}\subset \Q$ with positive $\hat{\mu}$-measure such that for every $Q(x)\in\hat{\Lambda}$ the measure $\mu_x$ is not atomic. Moreover, by the invariance of the disintegration, $\hat{\Lambda}$ can be assumed to be invariant and, by the ergodicity, of full measure.\\

Let $Q=Q(x_0)\in\hat{\Lambda}\cap\supp(\hat{\mu})$ and let $\B$ be a foliated (by $\F^1$) box around $Q$. That is, some image of a topological embedding 
\[
 \phi: D^{3}\times D^1\to \T^4,
\]
where $D^k$ is the closed unit disk in $\R^k$ and, such that every plaque $P_x=\phi(\{x\}\times D^1)$ is contained in a leaf of $\F^1$. Let us identified $\B$ with the product $D^{3}\times D^1$ through the corresponding homeomorphism. Let $\hat{V}$ be an open neighborhood of $Q$ small enough so it is contained in $\B$. Moreover, since $\tH^{-1}(\tx)$ is uniformly bounded we can assume that $\B$ contains $\P(x)$ for every $x\in D^3$. \\

Consider the following map 
\begin{eqnarray*}
\psi: D^3 \times [0,1] & \rightarrow & \B\\
 (x,t) &\mapsto & (x,\theta_x(t))
\end{eqnarray*}
where $(x,\theta_x(t))$ is defined as the higher point in the local leaf $Q(x)\subset \B$ such that $\mu_x([x,\theta_x(t)]_1)=t$.\\

Notice that $\psi$ is an invertible map when restricted to its image. Moreover, since we are assuming a non atomic disintegration, $\psi^{-1}$ is a continuous map restricted to the second coordinate and a measurable map when restricted to the first coordinate. Maps of these type are known as Caratheodory functions and these are measurable maps (\cite[Lemma 4.51]{Aliprantis}).\\

Consider the set $H_t^0:=\psi\left(\Sigma \times [0,t]\right)$, which is measurable since $\psi^{-1}$ is Caratheodory. Thus, the set  $H_t=\cup_{n\in\Z}f^n(H_t^0)$ forms an invariant measurable set. Notice that by the definition of $\psi$ we have that if $0<t<1$
\begin{align*}
 \mu(H_t)&=\int \mu_x(H_t\cap Q(x))d\hat{\mu}(Q(x)),\\
         &\geq \int_{\hat{R}}\mu_x(H_t^0\cap Q(x))d\hat{\mu}(Q(x)),\\
         &= \int_{\hat{R}}\mu_x([x,\theta_x(t)]_1)d\hat{\mu}(Q(x)),\\
         &=\hat{\mu}(\hat{R})t>0. 
\end{align*}
On the other hand, define $G_t^0=\left(H_t^0\right)^c$ and the $f$-invariant set $G_t=\cup_{n\in\Z}f^n(G_t^0)$. In a similar manner as before we have
\[
 \mu(G_t)\geq \hat{\mu}(\hat{R})(1-t)>0.
\]
Therefore, by the ergodicity, both sets should have full measure. Although, this would imply that their intersection also should have full measure but we claim this is not the case. In fact, if it were true, for $\hat{\mu}$-almost every $Q(x)$ 
\[
 \mu_x(H_t\cap G_t\cap Q(x))=1.
\]
But if $\omega$ belongs to $H_t\cap G_t\cap Q(x)$, without loss of generality, we may assume that for some $n\in\N$, $\omega\in f^{-n}(H_t^0)\cap G_t^0$. Hence, since $f$ preserves orientation, it is easy to see that
\[  
 t<\mu_{\omega}([0_{\omega},\omega]_1)=(f^n)_{\ast}\mu_{\omega}\left(f^n([0_{\omega},\omega]_1)\right)\leq \mu_{f^n(\omega)}\left(\left[0_{f^n(\omega)},f^n(\omega)\right]_1\right)\leq t.
\]
This is an absurd, which implies that the disintegration of $\mu$ is atomic for $\hat{\mu}$-almost every point.
\end{proof}

We have proved that $\mu$ has atomic disintegration with respect the partition $\Q$. We now want to see that there is a finite number of atoms on the disintegration considered. In order to do that we first need to prove the measurability of certain sets.\\

Consider $\B$ a foliated (by $\F^1$) box, as before, and identify $\B$ with the product $D^{3}\times D^1$ through the corresponding homeomorphism.\\

Fix $\delta>0$, and consider the set 
\[
 H_\delta = \{ x \in \B| \mu_x(\{x\}) \geq \delta \}.
\]
Let us see that this is a measurable set. To do so, consider a countable basis $\mathcal{V}$ of the topology of $\T^4$. From Rokhlin's theorem we know that the map $x\mapsto \mu_x(V)$ is measurable (up to measure zero) for any measurable set $V$. Therefore, by Lusin's theorem, given any $\varepsilon>0$ there exist a compact set $K_{\varepsilon}\subset D^3$ such that $\hat{\mu}(K_{\varepsilon})>1-\varepsilon$ and $x\mapsto\mu_x(V)$ is continuous on $K_{\varepsilon}$, for every $V\in\mathcal{V}$. In particular, $x\mapsto\mu_x$ is continuous with respect to the weak* topology for any $x$ in $K_{\varepsilon}$.\\

Let $\varepsilon>0$ be fixed. For each $x\in C$, let $A(x)$ be the set of atoms of $\mu_x$. It is clear that the set
\begin{equation}\label{eqn.DefGamma}
 \Gamma_{\delta}(x):=\{a\in A(x):\mu_x(\{a\})\geq \delta\},
\end{equation}
is finite, and hence compact. Furthermore, the definition of $K_{\varepsilon}$ ensures that the function $x\mapsto\Gamma_{\delta}(x)$ is upper semi-continuous on $x\in K_{\varepsilon}$. Therefore,
\[
 \Gamma(\varepsilon,\delta):=\{(x,a):x\in K_{\varepsilon}\text{ and }a\in\Gamma_{\delta}(x)\},
\]
is a closed set. Then, $\cup_{n}\Gamma(1/n,\delta)$ is a (measurable) full measure subset of $H_\delta$. Thus, $H_\delta$ is a measurable set (up to measure zero).\\

Consider $\{\B_k:k\in\N\}$ a countable cover of $\T^4$ by foliated boxes. Proceeding as before we obtain the measurable sets $H_\delta^k$ of atoms of measure bigger or equal to $\delta$ in each foliated box $\B_k$. Therefore, 
\[
 H_{\delta}^+:=\bigcup_{k\in\N}H_{\delta}^k=\{x\in C:\mu_x(\{x\})\geq\delta\}
\] 
is also measurable.\\

\begin{lemma}\label{lem.OneAtomQ}
 $\hat{\mu}$-almost every $Q(x)$ contains only one atom.
\end{lemma}
 
\begin{proof}
Let $x\in \M$ and $\delta\geq 0$. Consider the set $H_{\delta}^+$ as before and notice that 
\[
 \delta\leq\mu_x(\{x\})\leq f_{\ast}\mu_x(\{f(x)\})=\mu_{f(x)}(\{f(x)\}).
\]
Therefore, $H_{\delta}^+$ is invariant and, by ergodicity, it has measure zero or one. We know that $\mu(H_1^+)=0$ and $\mu(H_0^+)=1$. Let $\delta_0$ be the discontinuity point of the function $\delta \mapsto \mu(H_\delta^+)$, for $\delta \in [0,1]$. Hence $\mu(H_{\delta_0}^+)=1$, that means the weight of the atoms are all equal to $\delta_0$. Therefore there are $n=1/\delta_0$ atoms on each element of the partition $\Q$.\\

Let us see that the disintegration of $\mu$ on $\Q$ has one atom per local leaf. Assume by contradiction that $n=2$, as the case of finite atoms is similar. Let $a(x)$ and $b(x)$ be the two atoms of $\mu_x$. Without loss of generality, let us assume that $a(x)<b(x)$, where ``$<$'' is the fixed order in $\F^1$. Consider 
\[
 A:=\{a(x): x\in C\}\quand B:=\{b(x): x\in C\},
\]
the sets of first and second atoms respectively. Since $f$ preserves the orientation in $\F^1$, it is easy to see that $A$ and $B$ are invariant sets.\\

Let $Q=Q(x_0)\in\supp\hat{\mu}$ and let $\hat{V}$ be an open neighborhood of $Q$. Consider the disjoint sets 
\[
 B(a):=\bigcup_{Q(x)\in \hat{V}}\{x\}\times B(a(x))\quand B(b):=\bigcup_{Q(x)\in \hat{V}}\{x\}\times B(b(x)),
\]
where $B(a(x))$ and $B(b(x))$ are two disjoint closed balls in $Q(x)$ around $a(x)$ and $b(x)$ respectively.  Notice that, following the proof of the measurability of $H_{\delta}^+$ by substituting the set $\Gamma_{\delta}(x)$ in \eqref{eqn.DefGamma} by $B(a(x))$, we can prove that $B(a)$ and $B(b)$ are both measurable sets. By the definition of $B(a)$ and $B(b)$, their saturation by $\Q$ coincide, that is, $\pi(B(a)) = \pi(B(b))$. Therefore, $B(a)$ and $B(b)$ have positive $\mu$-measure.\\

Let us define the $f$-invariant sets
\[
 H(a):=\bigcup_{n\in\Z}f^n(B(a))\quand H(b):=\bigcup_{n\in\Z}f^n(B(b)).
\]
We claim that $\mu(H(a)\cap H(b))=0$. In fact, if it is not true we have that
\[
0<\mu(H(a)\cap H(b))=\int\mu_x(H(a)\cap H(b)\cap Q(x))d\hat{\mu}.
\]
Therefore, there must exist $\hat{\Lambda}\subset \hat{C}$ of positive $\hat{\mu}$-measure such that for every $Q(x)\in\hat{\Lambda}$ 
\[
 \mu_x(H(a)\cap H(b)\cap Q(x))>0.
\]
Hence, $a(x)$ or $b(x)$ must belong to the intersection of $H(a)\cap H(b)$. Without loss of generality, let us assume that there exists $n\in\Z$ such that $a(x)\in f^n(B(b))\cap B(a)$. Therefore, we have that $f^{-n}(a(x))=b(y)$ for some $Q(y)\in\hat{V}$. However, this contradicts the invariance of $A$, and we proved our claim.\\

Now, by ergodicity of $\mu$, the sets $H(a)$ and $H(b)$ should have full measure and have zero measure intersection. Absurd, therefore we have only one atom on $\Q(x)$ which proves our claim. \\
\end{proof}

Let denote the atom found in Lemma \ref{lem.OneAtomQ} by $a(x)$, that is 
\begin{equation}\label{eq.Etaxz}
 \mu_x=\delta_{a(x)}.
\end{equation}
We now want to see that the disintegration of $\mu$ on $\P$ has only one atom in each connected component of every element of the partition.\\

Recall that $\widehat{C}:=C/\Q$. Define $\hat{f}:\widehat{C}\to\widehat{C}$ by $\hat{f}(\hat{z}):=\widehat{f(z)}$, which satisfies 
\[
 \pi\circ f=\hat{f}\circ\pi.
\]
Notice that by \eqref{eqn.UnionRectangles} we can identify $\widehat{\P(x)}$ with $n_x$ connected components in the $\F^2$ foliation. That means the space $\widehat{C}$ has now a one dimensional foliation coming from this quotient. Consider the partition $\widehat{\Q}$ given by
\[
 \widehat{\Q}:=\{\widehat{R(x)}: x\in C\},
\]
where $R(x)$ is the rectangle in $\P(x)$ containing $x$. Moreover, notice that $\widehat{R(x)}$ can be identified with the interval $[c_0(x),c_1(x)]_2$, where $c_0(x)$ and $c_1(x)$ are the corners of $R(x)$ in the same $\F^2$-leaf. Consequently, proceeding as before, the conditional measures $\hat{\eta}_x$ defined by the partition $\widehat{\Q}$ for the measure $\hat{\mu}$, have at most one atom in each $\widehat{R(x)}$ that we denote by $\hat{a}(x)$. Thus,
\begin{equation}\label{eqn.Etax}
 \hat{\eta}_x=\delta_{\hat{a}(x)}.
\end{equation}
Combining this with \eqref{eq.Etaxz} we have that $a_j(x)\in\pi^{-1}\left(\hat{a}(x)\right)\cap R_j(x)$ is the only one atom per rectangle $R_j(x)$. This concludes the proof of our result.\\

\begin{remark}
For the higher dimensional case the proof follows in a similar manner. More precisely, if $dim E^c=k$ after proving the atomic decomposition of the conditional measures of $\mu$ defined by the partition $\Q$, we proceed by reducing the dimension to $k-1$ using a quotient procedure to define the conditional measures \eqref{eqn.Etax}. Continuing in this manner we keep reducing the dimension until getting dimension 1 and conclude our proof.
\end{remark}

\section{Proof of Corollary \ref{cor.main}}\label{sec.ProofCor}

We are left with the task of proving that if $H$ sends center leaves of $f$ to center leaves of $A$ and if one of the  conditions \ref{it.2} or \ref{it.3} are satisfied, then $\mu$ is virtually hyperbolic. 

First, let us assume \ref{it.2} is valid. Moreover, we assume that the center direction of $A$ is expanding, otherwise we work with $f^{-1}$.\\

By the proof of Theorem \ref{main:ergodic}, there are at most countably many elements in $\P$ with positive measure, we get a full measurable subset $\mathcal{M}\subset\mathbb{T}^4$ which intersects each center leaf in at most countably many points. Furthermore, we claim that there exist finitely many atoms of $\mu$ per (global) center leaf. In fact, this was proved in \cite[Proposition~3.2]{CrisostomoTahzibi}. Although they assume one dimensional center foliation, under our assumptions their proof could be applied. Let us recall the main steps. \\

Assume by contradiction that every full measurable subset of $\mathcal{M}$ intersects any typical center leaves in infinitely many points. Define $\nu=H_{\ast}\mu$ which is an invariant measure by the linear hyperbolic automorphism. Let $R_i$ be the Markov partition for $A$ and consider the partition $\Q:=\{\F^c_R(x): x\in R_i \text{ for some }i\}$, where $\F^c_R(x)$ denotes the connected component of $\F^c(x)\cap R_i$ containing $x$. The partition $\Q$ is measurable and we denote $\nu_x$ the disintegration of $\nu$ along the elements of $\Q$. The assumption of full support of $\nu$ guarantees it gives zero mass to the boundary of the Markov partition.\\

As $H(\mathcal{M})$ intersects typical leaves in a countable number of points, $\nu_x$ must be atomic. Moreover, there exists a natural number $\alpha_0\in\mathbb{N}$ such that $\nu_x$ contains exactly $\alpha_0$ atoms for $\nu$-almost every $x$ (see \cite[Lemma~3.3]{CrisostomoTahzibi}). Hence, given a fixed $L\in\mathbb{R}_+$, there exist $N\in\mathbb{N}$ such that the number of atoms in any typical center plaque of diameter $L$ is at most $N$. We are assuming that $H(\mathcal{M})$ intrinsically intersects center leaves in infinitely many points (or non uniformly finite). Taking $D\subset\mathcal{F}^c(x)$ with more than $N$ atoms. By backward contraction along central leaves by $A$ there exists $n>0$ such that the diameter of $A^{-n}(D)$ is less that $L$. As $\nu$ is invariant and the uniqueness of the disintegration, we get a center plaque with diameter less than $L$ containing more than $N$ atoms, which is absurd and establishes our claim.\\

We have proved that the number of atoms is finite and constant by ergodicity on almost every center leaf. The task is now to conclude that, since $f$ preserves orientation, the number of atoms is one. In order to do this first consider the set of atoms in each $\F^1$-leaf. Proceeding as in the proof of Lemma \ref{lem.OneAtomQ}, we can prove that there must be only one atom per $\F^1$-leaf.  Now consider the space $\widetilde{C}:= C/\sim$, where $ x \sim y$ iff $y \in \mathcal \F^1(y)$. The way we should see $\tilde{C}$ is as turning the center foliation (which is a plane) into a 1-dimensional segment. Let us denote this new foliation as $\tilde{\Q}$. Notice that the disintegration of $\mu$ in the partition given by $\tilde{\Q}$ is exactly the quotient measure $\hat \eta^ x$.\\

Since $\mathcal \F^1$ has an orientation, we may define a transversal orientation by the following way: a vector $v \in T_x \F_{loc}^c(x)$ points in the positive direction if for any positive vector $w \in T_x \F_{loc}^c(x)$ we have $\omega_x(v,w) >0$, where $\omega_x$ is the restriction of the volume form to $\F_{loc}^c(x)$. \\

Now consider the extremal atoms per central leaf. By left extremal atom we consider the atom whose projection by the map $\pi:C\to\widetilde{C}$ is the left extreme one (see Figure \ref{fig.GlobalAtom}). Since $f$ preserves the orientation in $\F^1$, then $f$ preserves the transversal orientation. Once again, proceeding as the proof of Lemma \ref{lem.OneAtomQ} we conclude that there is only one atom per global center leaf. Therefore, $\mu$ is virtually hyperbolic.\\

On the other hand, if \ref{it.3} is valid then $H(\F^1)$ must coincides with $E^s_A$ or $E_A^u$. Without loss of generality, let us assume $H(\F^1)$ coincides with $E_A^u$. By the proof of Theorem \ref{main:ergodic}, the set $\mathcal{M}$ intersects each  $\F^1$ leaf in at most countably many points. Proceeding as before, using the $\F^1$ foliation instead the center one, one can also conclude that $\mu$ is virtually hyperbolic.

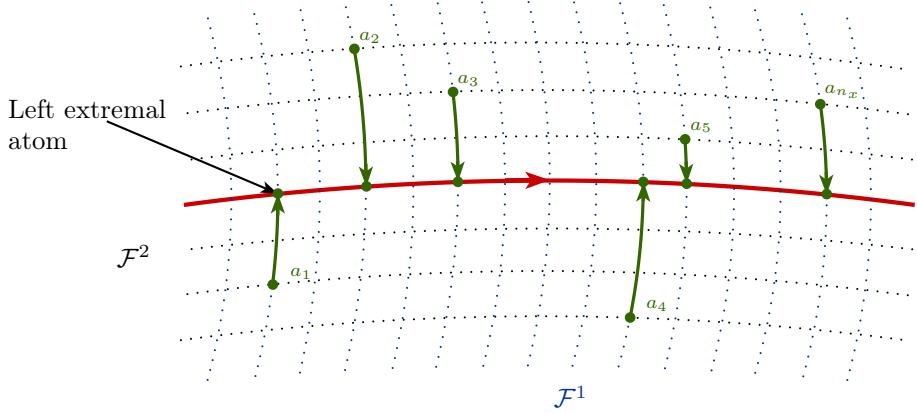
\begin{figure}
\centering
\captionsetup{justification=centering}

\newrgbcolor{ccqqqq}{0.8 0 0}
\newrgbcolor{qqttzz}{0 0.2 0.6}
\newrgbcolor{ttwwqq}{0.2 0.4 0}

\psset{xunit=0.6cm,yunit=0.6cm,algebraic=true,dotstyle=o,dotsize=3pt 0,linewidth=0.8pt,arrowsize=3pt 2,arrowinset=0.25}

\begin{pspicture*}(-6.5,-8)(19.56,7.23)

\parametricplot[linestyle=dotted]{1.438933038681401}{1.7070124324631315}{1*59.92*cos(t)+0*59.92*sin(t)+6.34|0*59.92*cos(t)+1*59.92*sin(t)+-62.39}
\parametricplot[linestyle=dotted]{1.4366131522745456}{1.7048701232109778}{1*59.88*cos(t)+0*59.88*sin(t)+6.35|0*59.88*cos(t)+1*59.88*sin(t)+-61.35}
\parametricplot[linestyle=dotted]{1.4545880574749945}{1.6891027161937202}{1*68.47*cos(t)+0*68.47*sin(t)+6.52|0*68.47*cos(t)+1*68.47*sin(t)+-69}
\parametricplot[linewidth=1.6pt,linecolor=ccqqqq]{1.4362722835818533}{1.7053205450573163}{1*59.72*cos(t)+0*59.72*sin(t)+6.48|0*59.72*cos(t)+1*59.72*sin(t)+-59.19}
\parametricplot[linestyle=dotted]{1.4398908084292725}{1.7045415640600527}{1*60.7*cos(t)+0*60.7*sin(t)+6.53|0*60.7*cos(t)+1*60.7*sin(t)+-59.19}
\parametricplot[linestyle=dotted]{1.4369775988146314}{1.7047383570510684}{1*59.99*cos(t)+0*59.99*sin(t)+6.36|0*59.99*cos(t)+1*59.99*sin(t)+-57.46}
\parametricplot[linestyle=dotted]{1.4275639066583594}{1.7112163309832855}{1*56.64*cos(t)+0*56.64*sin(t)+6.13|0*56.64*cos(t)+1*56.64*sin(t)+-53.06}
\parametricplot[linestyle=dotted,linecolor=qqttzz]{-0.2589490416790463}{0.2561277029838513}{1*15.73*cos(t)+0*15.73*sin(t)+-16.22|0*15.73*cos(t)+1*15.73*sin(t)+0.13}
\parametricplot[linestyle=dotted,linecolor=qqttzz]{-0.24847858790406097}{0.2533679762640946}{1*16.16*cos(t)+0*16.16*sin(t)+-15.63|0*16.16*cos(t)+1*16.16*sin(t)+0.16}
\parametricplot[linestyle=dotted,linecolor=qqttzz]{-0.21560038250441327}{0.21120664301181347}{1*18.97*cos(t)+0*18.97*sin(t)+-17.52|0*18.97*cos(t)+1*18.97*sin(t)+0.32}
\parametricplot[linestyle=dotted,linecolor=qqttzz]{-0.22569590480210344}{0.22017112155274524}{1*18.19*cos(t)+0*18.19*sin(t)+-15.73|0*18.19*cos(t)+1*18.19*sin(t)+0.39}
\parametricplot[linestyle=dotted,linecolor=qqttzz]{-0.2299891503973468}{0.23557478716500682}{1*17.44*cos(t)+0*17.44*sin(t)+-13.94|0*17.44*cos(t)+1*17.44*sin(t)+0.34}
\parametricplot[linestyle=dotted,linecolor=qqttzz]{-0.23475855500593656}{0.22522388125080922}{1*17.67*cos(t)+0*17.67*sin(t)+-13.19|0*17.67*cos(t)+1*17.67*sin(t)+0.5}
\parametricplot[linestyle=dotted,linecolor=qqttzz]{-0.2406942442470834}{0.24560912946714195}{1*16.73*cos(t)+0*16.73*sin(t)+-11.22|0*16.73*cos(t)+1*16.73*sin(t)+0.4}
\parametricplot[linestyle=dotted,linecolor=qqttzz]{-0.2368613212533699}{0.23192095946279598}{1*17.34*cos(t)+0*17.34*sin(t)+-10.85|0*17.34*cos(t)+1*17.34*sin(t)+0.49}
\parametricplot[linestyle=dotted,linecolor=qqttzz]{-0.25504126325191834}{0.244983345635518}{1*16.28*cos(t)+0*16.28*sin(t)+-8.77|0*16.28*cos(t)+1*16.28*sin(t)+0.52}
\parametricplot[linestyle=dotted,linecolor=qqttzz]{-0.2644294392483273}{0.25447439575352954}{1*15.69*cos(t)+0*15.69*sin(t)+-7.15|0*15.69*cos(t)+1*15.69*sin(t)+0.5}
\parametricplot[linestyle=dotted,linecolor=qqttzz]{-0.23990094560670006}{0.24518545590870444}{1*16.76*cos(t)+0*16.76*sin(t)+-7.27|0*16.76*cos(t)+1*16.76*sin(t)+0.34}
\parametricplot[linestyle=dotted,linecolor=qqttzz]{-0.24954157024380397}{0.2550438078858211}{1*16.11*cos(t)+0*16.11*sin(t)+-5.58|0*16.11*cos(t)+1*16.11*sin(t)+0.29}
\parametricplot[linestyle=dotted,linecolor=qqttzz]{-0.27735663855019776}{0.2768195775981685}{1*14.69*cos(t)+0*14.69*sin(t)+-3.14|0*14.69*cos(t)+1*14.69*sin(t)+0.28}
\parametricplot[linestyle=dotted,linecolor=qqttzz]{-0.25942436688335313}{0.2651887834136855}{1*15.48*cos(t)+0*15.48*sin(t)+-2.92|0*15.48*cos(t)+1*15.48*sin(t)+0.15}
\parametricplot[linestyle=dotted,linecolor=qqttzz]{-0.25925503532936034}{0.25861054882642875}{1*15.65*cos(t)+0*15.65*sin(t)+-2.12|0*15.65*cos(t)+1*15.65*sin(t)+0.11}
\parametricplot[linewidth=1.2pt,linecolor=ttwwqq]{-0.1197877952429538}{0.0049240162956248066}{1*16.16*cos(t)+0*16.16*sin(t)+-15.63|0*16.16*cos(t)+1*16.16*sin(t)+0.16}
\parametricplot[linewidth=1.2pt,linecolor=ttwwqq]{8.409078352180899E-4}{0.16892944583905878}{1*18.19*cos(t)+0*18.19*sin(t)+-15.73|0*18.19*cos(t)+1*18.19*sin(t)+0.39}
\parametricplot[linewidth=1.2pt,linecolor=ttwwqq]{2.1025287395711277E-5}{0.11292633657533116}{1*17.67*cos(t)+0*17.67*sin(t)+-13.19|0*17.67*cos(t)+1*17.67*sin(t)+0.5}
\parametricplot[linewidth=1.2pt,linecolor=ttwwqq]{-0.1921390198810604}{3.514068882600108E-4}{1*15.69*cos(t)+0*15.69*sin(t)+-7.15|0*15.69*cos(t)+1*15.69*sin(t)+0.5}
\parametricplot[linewidth=1.2pt,linecolor=ttwwqq]{0.007043515370644108}{0.06566601556149111}{1*16.76*cos(t)+0*16.76*sin(t)+-7.27|0*16.76*cos(t)+1*16.76*sin(t)+0.34}
\parametricplot[linewidth=1.2pt,linecolor=ttwwqq]{0.0048587683820751305}{0.1341039757655995}{1*15.48*cos(t)+0*15.48*sin(t)+-2.92|0*15.48*cos(t)+1*15.48*sin(t)+0.15}

\psline[linewidth=1.8pt,linecolor=ccqqqq]{->}(5.51,0.53)(6.49,0.54)
\psline[linewidth=1.2pt,linecolor=ttwwqq]{->}(0.53,-0.24)(0.53,0.24)
\psline[linewidth=1.2pt,linecolor=ttwwqq]{->}(2.46,0.9)(2.47,0.4)
\psline[linewidth=1.2pt,linecolor=ttwwqq]{->}(4.47,0.95)(4.48,0.51)
\psline[linewidth=1.2pt,linecolor=ttwwqq]{->}(8.54,0.08)(8.55,0.5)
\psline[linewidth=1.2pt,linecolor=ttwwqq]{->}(9.48,0.95)(9.49,0.46)
\psline[linewidth=1.2pt,linecolor=ttwwqq]{->}(12.55,0.72)(12.56,0.22)
\psline{->}(-3.22,1.84)(0.53,0.24)

\rput[tl](6.58,-4){\qqttzz{$ \mathcal{F}^1 $}}
\rput[tl](-3,-0.9){$ \mathcal{F}^2 $}
\rput[lt](-5.38,2.33){\parbox{2.98 cm}{Left extremal \\ atom}}

\begin{scriptsize}
\psdots[dotsize=4pt 0,dotstyle=*,linecolor=ttwwqq](0.42,-1.77)
\rput[bl](0.8,-1.64){\ttwwqq{$a_1$}}
\psdots[dotsize=4pt 0,dotstyle=*,linecolor=ttwwqq](2.21,3.44)
\rput[bl](2.31,3.58){\ttwwqq{$a_2$}}
\psdots[dotsize=4pt 0,dotstyle=*,linecolor=ttwwqq](4.37,2.49)
\rput[bl](4.47,2.63){\ttwwqq{$a_3$}}
\psdots[dotsize=4pt 0,dotstyle=*,linecolor=ttwwqq](8.25,-2.5)
\rput[bl](8.6,-2.35){\ttwwqq{$a_4$}}
\psdots[dotsize=4pt 0,dotstyle=*,linecolor=ttwwqq](9.46,1.44)
\rput[bl](9.57,1.6){\ttwwqq{$a_5$}}
\psdots[dotsize=4pt 0,dotstyle=*,linecolor=ttwwqq](12.42,2.22)
\rput[bl](12.51,2.36){\ttwwqq{$a_{n_x}$}}

\psdots[dotsize=4pt 0,dotstyle=*,linecolor=ttwwqq](0.53,0.24)
\psdots[dotsize=4pt 0,dotstyle=*,linecolor=ttwwqq](2.47,0.4)
\psdots[dotsize=4pt 0,dotstyle=*,linecolor=ttwwqq](4.48,0.5)
\psdots[dotsize=4pt 0,dotstyle=*,linecolor=ttwwqq](8.54,0.5)
\psdots[dotsize=4pt 0,dotstyle=*,linecolor=ttwwqq](9.49,0.46)
\psdots[dotsize=4pt 0,dotstyle=*,linecolor=ttwwqq](12.56,0.23)
\end{scriptsize}

\end{pspicture*}

\vspace{-2cm}
\caption{Global Atom in a center leaf.}\label{fig.GlobalAtom}
\end{figure}


\section{Proof of Theorem \ref{th.Equilibrium}}\label{sec.ProofEq}

\textbf{Case 1:} $\mu(C)=0$\\

This case follows as in the proof of \cite[Theorem~A]{CrisostomoTahzibi}. We present it here for completeness.\\

Let $\nu$ be the unique equilibrium state of $(A,\phi)$. Assume by contradiction that there exist $\eta$ another equilibrium state for $(f,\varphi=\phi\circ H)$. By the uniqueness of $\nu$ we have that $H_{\ast}\mu=H_{\ast} \eta$. Let $\psi:\T^4\to \R$ be any continuous map. Since $H^{-1}H(C)=C$ we have  $\eta(C)=0$. Therefore,
\begin{align*}
\int \psi d\mu &= \int_{\T^4\backslash C}\psi d\mu, \\
&=\int_{\T^4\backslash C}\psi\circ H^{-1} dH_{\ast}\mu,\\
&=\int_{\T^4\backslash C}\psi\circ H^{-1} dH_{\ast}\eta,\\
&=\int_{\T^4\backslash C}\psi d\eta,\\
&=\int \psi d\eta.
\end{align*}
Since $\psi$ is arbitrary, this implies that $\mu=\eta$.\\

\textbf{Case 2:} $\mu(C)=1$\\

Consider the partition:
\[
 \P:=\{\P(x):=H^{-1}H(x)|x\in C\},
\] 
and denote by $\mu_x$ the conditional measure of $\mu$ supported on $\P(x)$. We proceed as in the proof of Theorem \ref{main:ergodic}. Hence, we have that 
\[
 \mu_x=\sum_{j=1}^{n_x}p_j(x)\delta_{a_j(x)},
\]
for some $a_j(x)\in R_j(x)$. Moreover, if conditions \ref{it.2} and \ref{it.3} are satisfied, then $\mu$ is virtually hyperbolic. The only thing left to prove is the existence of another equilibrium state.\\

\begin{lemma}\label{lem.MeasurableExtremal}
 If $H$ sends center leaves of $f$ to center leaves of $A$ and \ref{it.Eq3} is satisfied then, the set of extremal points of intervals $Q(x)=\P(x)\cap\F^1(x)$ forms a measurable set.
\end{lemma}

\begin{proof}
Let us denote $\F^1_A$ the foliation of the center direction of $A$ induced by the image of $\F^1$ by the semiconjugacy $H$. We will prove the measurability of the lower extremal points of $Q(x)$. The case of higher extremal points is similar.\\

Consider $\varphi:\T^4\to\T^4$ the flow on $\T^4$ having constant speed one in $\F^1_A$. More precisely, we know that the leaves of $\F^1$ in the center foliation of A are straight lines and orientable by assumption. Define $\varphi(t,x)$ the unique point in the $\F^1_A(x)$ which has distance $t$ inside this $\F^1_A$-leaf and in the positive direction from $x$.\\

Following the proof of \cite[Lemma~ 3.2]{PonceTahzibiVaraoBernoulli}, we have that $H(C)$ is a measurable set. Therefore, $\varphi(-1/n,H(C))$ is a measurable set. Furthermore, since $H$ is continuous, the set $H^{-1}\left(\varphi(-1/n,H(C))\right)$ is also measurable.\\

Consider $\hat{C}=C/\Q$ where $\Q:=\{Q(x):=\P(x)\cap\F^1(x):x\in C\}$. Let 
\[
 \phi_n:\hat{C}\to H^{-1}\left(\varphi(-1/n,H(C))\right),
\]
be the function given by the Measurable Choice Theorem \ref{th.Measurable} applied to the product $\hat{C}\times \T^4$ and the measurable set $G=H^{-1}\left(\varphi(-1/n,H(C))\right)$.\\

Notice that fixing $Q(x)\in\hat{C}$ we have that $\phi_n(Q(x))$ is an increasing sequence. Therefore, we can define the function
\begin{align*}
\phi:&\hat{C}\to \T^4\\
     &Q(x)\mapsto\lim_{n\to\infty}\phi_n(Q(x)),
\end{align*}
and by its construction $\phi(Q(x))$ is the lower extreme of $Q(x)$. Notice that $\phi$ is a measurable function because it is the limit of measurable functions. Let $\pi$ the canonical projection and let $\hat{\mu}=\pi_{\ast}\mu$ the measure in the quotient space. By Lusin's theorem for any $n\in\N$ there exist a compact set $\hat{K}_n\subset\hat{C}$ such that $\hat{\mu}(\hat{K}_n^c)<1/n$ and $\phi$ is a continuous function when restricted to $\hat{K}_n$. Therefore, $\phi(\hat{K}_n)$ is a compact set. Without loss of generality we may consider $\hat{C}=\cup_{n\in\N}\hat{K}_n$. Therefore,
\[
 \phi(\hat{C})=\bigcup_{n\in\N}\phi(K_n),
\]
is a measurable set. Thus, we have proven so far that the base of the intervals from $\Q$ forms a measurable set.
\end{proof}

We have seen that the center foliation is measure theoretically equivalent to the partition of $\T^4$ into points, hence measurable. Let us denote $(\hat{M},\hat{\mu})$ the quotient space $\hat{M}:=\T^4/\F^c$ equipped with the quotient measure. We denote by $\hat{f}:\hat{M}\to\hat{M}$ the induced map on the quotient space. Therefore, since $\mu$ if $f$-invariant, then $\hat{\mu}$ is $\hat{f}$-invariant.\\

Notice that, by the virtual hyperbolicity proved above, every element $\hat{x}\in\hat{M}$ can be identified by the unique $\Q_x(z)\subset\F^c(x)$ where its atom belongs to. When $\Q_x(z)$ is a collapse interval inside a $\F^1$-leaf, we define $\Q(\hat{x}):=\Q_x(z)$. On the other hand, if $\Q_x(z)$ is a point, this means that the rectangle $R_j$ containing the atom is one dimensional and contained in an $\F^2$-leaf. In this case, we define $\Q(\hat{x}):=R_j$ (see Figure \ref{fig.RecDim1F2}).\\

Thus, we can write
\[
 \mu=\int\delta_{a(\hat{x})}d\hat{\mu},
\]
where $a(\hat{x})$ is the atom inside the collapse interval $\Q(\hat{x})$. Choose $b(\hat{x})\neq a(\hat{x})$ the left (or right) extreme point of $\Q(\hat{x})$. Let us define
\[
 \eta=\int\delta_{b(\hat{x})}d\hat{\mu},
\]
which is well-defined because $\{b(\hat{x}):\hat{x}\in\hat{M}\}$ is measurable by Lemma \ref{lem.MeasurableExtremal}. We claim that this is an $f$-invariant ergodic measure satisfying $H_{\ast}\eta=H_{\ast}\mu$. In order to see this, consider any continuous map $\psi$ and notice that
\begin{align*}
 \int\psi\circ fd\eta &=\int\int\psi\circ f d\delta_{b(\hat{x})}d\hat{\mu}\\
                     &=\int\psi(f(b(\hat{x})))d\hat{\mu}\\
                     &=\int\psi(b(\hat{f}(\hat{x})))d\hat{\mu}\\
                     &=\int\psi(b(\hat{x}))d\hat{\mu}\\
                     &=\int\psi d\eta,\\
\end{align*}
where the third equality comes from the invariance of collapse intervals and that $f$ preserves the orientation of the $\F^i$-foliations with $i=1,2$. The fourth equality is due to the $\hat{f}$-invariance of $\hat{\mu}$.\\

To see the ergodicity of $\eta$, consider any invariant subset $D$ with positive $\eta$-measure. Since $\hat{\mu}$ is ergodic and $f(b(\hat{x}))=b(\hat{f}(\hat{x}))$, we have that the set $\{\hat{x}:\mathbbm{1}_D(b(\hat{x}))=1\}$ is $\hat{f}$-invariant. So the ergodicity of $\hat{\mu}$ guarantees it has full measure, which implies $\eta(D)=1$.\\

Notice that, if $\varphi=\phi\circ H$, since $H(a(\hat{x}))=H(b(\hat{x}))$ then
\[
 \int\varphi d\eta=\int\varphi(b(\hat{x})) d\hat{\mu}=\int\varphi(a(\hat{x})) d\hat{\mu}=\int\varphi d\mu.
\]
However, by the essential uniqueness of disintegration we have that $\eta\neq\mu$.\\

We are left with the task of determining $h_{\eta}(f)=h_{\mu}(f)$. But this is a direct consequence of the fact that $(f,\mu)$ and $(f,\eta)$ are measure theoretically isomorphic by the map that sends $a(\hat{x})$ to $b(\hat{x})$. Thus, $\eta$ is also an equilibrium state form $(f,\varphi)$.\\

\section*{Acknowledgments}
The authors thank Cristina Lizana by the communication of her work and for clarifying some questions. C.F.A. was partially funded by CAPES-Brazil (grant \#2019/88882.329056-01). A.S. thanks the Math Department of ICMC (São Carlos) where most of the work was developed. This work was partially supported by FAPESP (Fundação de Amparo à Pesquisa do Estado de São Paulo), grant \#2018/18990-0 and Universidad de Costa Rica for A.S. and R.V. was partially supported by National Council for Scientific and Technological Development– CNPq, Brazil and partially supported by FAPESP (grants \#17/06463-3 and \#16/22475-9).
\bibliography{bibliography}
\bibliographystyle{plain}

\end{document}